\documentclass[a4paper]{amsart}
\usepackage[utf8]{inputenc}
\usepackage[english]{babel}
\usepackage{amsfonts}
\usepackage{amssymb,amsthm}
\usepackage[left=3cm, right=3cm, bottom=2cm, top=2cm]{geometry}

\usepackage{mathtools}

\usepackage{hyperref}
\usepackage{xcolor}
\hypersetup{
    colorlinks,
    linkcolor={blue!90!black},
    citecolor={green!70!black},
    urlcolor={blue!90!black}
}

\newcommand{\N}{\mathbb N}
\newcommand{\Q}{\mathbb Q}

\newcommand{\C}{\mathbb C}

\newcommand{\mc}{\mathcal}

\renewcommand{\phi}{\varphi}

\renewcommand{\ge}{\geqslant}
\renewcommand{\le}{\leqslant}

\theoremstyle{plain}
\newtheorem{theorem}{Theorem}[section]
\newtheorem*{theorem*}{Theorem}
\newtheorem{lemma}[theorem]{Lemma}
\newtheorem{corollary}[theorem]{Corollary}
\newtheorem{proposition}[theorem]{Proposition}

\newtheorem*{question}{Question}
\theoremstyle{definition}
\newtheorem{remark}[theorem]{Remark}

\newtheorem{definition}[theorem]{Definition}
\newtheorem{example}[theorem]{Example}

\begin{document}

\date{}

\title{On face numbers of flag simplicial complexes}
\author{Yury Ustinovskiy}

\address{Department of Mathematics, Princeton University}
\email{yuryu@math.princeton.edu}

\begin{abstract}
Denham and Suciu in~\cite{de-su-07} and Panov and Ray in~\cite{pa-ra-08} computed ranks of homotopy groups and Poincar\'e series of a moment-angle-complex $\mc Z(\mc K$) / Davis-Januzskiewicz space $DJ(\mc K)$ associated to a flag simplicial complex $\mc K$. In this note we revisit these results and interpret them as polynomial bounds on the face numbers of an arbitrary simplicial flag complex.
\end{abstract}
\maketitle

\section*{Introduction}
Let $\mc K$ be an $n$-dimensional simplicial complex and denote by $f_i$ the number of $i$-dimensional simplices of $\mc K$. Characterization of possible $f$-vectors $(f_0,\dots,f_n)$ of various classes of simplicial complexes is a classical problem of enumerative combinatorics. We mention several results in this direction: 
\begin{enumerate}
\item Kruskal-Katona theorem~\cite[II.1]{st-96}, describing all possible $f$-vectors of all simplicial complexes;
\item analogue of Kruskal-Katona theorem for Cohen-Macaulay simplicial complexes~\cite[II.2]{st-96};
\item Upper Bound Theorem due to McMullen~\cite[II.3.4]{st-96}, giving necessary conditions on $f$-vector for triangulations of $n$-dimensional spheres.
\item $g$-Theorem, characterizing $f$-vectors of simplicial polytopes, see~\cite[II.6.2]{st-96}.  
\end{enumerate}

Proofs of these results led to numerous constructions, associating to combinatorial objects (simplicial complexes, triangulations of spheres, polytopes, etc) certain algebraic and topological objects. These constructions allow to employ methods of homological algebra, algebraic geometry and algebraic topology in purely combinatorial problems.

In a similar spirit, in this note we derive a series of inequalities on $f$-vectors of a flag simplicial complex. Characterization of $f$-vectors of flag simplicial complex, or, equivalently \emph{clique vectors} of \emph{simple graphs}, is a well-studied problem with many partial results. In~\cite{zy-49} Zykov gave a generalization of Tur\'{a}n's theorem for graphs, Razborov~\cite{ra-08} proved asymptotic bounds on the component of $f_2$ in terms of $f_1$ and $f_0$. Herzog et al.~\cite{he-08} gave complete characterization for all possible $f$-vectors of 
\emph{chordal} flag simplicial complexes. Goodarzi~\cite{go-15} generalized this result for $k$-connected chordal simplicial complexes.

Our proof uses results of Denham, Suciu~\cite{de-su-07} and Panov, Ray~\cite{pa-ra-08}, where the authors relate Poincar\'e series of a face ring of a flag simplical complex to a Poincar\'e series of a free graded algebra. The main result can be formulated as follows

\begin{theorem}
Let $\mc K$ be a flag simplicial complex with $f$-vector $(f_0,\dots,f_n)$.  Then for any $N\ge 1$ we have
\begin{equation}
(-1)^{N}\sum_{d|N}\mu(N/d) (-1)^d p_d(\underline{\alpha})\ge 0,
\end{equation}
where $p_d$ is $d$-th Newton polynomial expressed in elementary symmetric polynomials $\underline{\alpha}=(\alpha_1,\alpha_2,\dots)$ with
\[
\alpha_n:=\sum_{i=0}^{n-1} f_i\binom{n-1}{i}.
\]
\end{theorem}
Here $\mu(n)$ is the M\"obius function
\[
\mu(n)=
\begin{cases}
(-1)^k, &  \mbox{if $n=p_1p_2\dots p_k$, is the prime decomposition of $n$ with distinct factors};\\
0, & \mbox{otherwise},
\end{cases}
\]
and the summation is taken over all positive integers $d$ dividing $N$.

The rest of the paper is organized as follows. In Section~\ref{sec:free_alg} we recall standard lower bounds on the dimensions of the graded components of a graded free algebra. In Section~\ref{sec:toric} we discuss basic notions of toric topology and recall results of Denhem-Suciu and Panov-Ray. Finally in Section~\ref{sec:main} we prove our main result.

\section{Free graded algebras}\label{sec:free_alg}

Throughout this paper we work with algebras over a field $k$ of characteristic zero, which are
\begin{enumerate}
\item (graded) $A=\sum^\infty_{i=0} A^i$ with  $A^i\cdot A^j\subset A^{i+j}$;
\item (locally finite dimensional) $\dim_k A^i<\infty$;
\item (commutative) $a\!\cdot\! b=(-1)^{ij}\,b\!\cdot\! a$ for $a\in A^i$, $b\in A^j$;
\item (connected) $A^0=\langle\mathbf 1\rangle_k$.
\end{enumerate}

The aim of this section is to derive certain identities involving dimensions of graded components of a free graded algebra $A^\bullet$. In different contexts similar identities appeared, e.g., in~\cite[Ch.\,3]{ba-86}, see also references therein. As a corollary of these identities we get lower bounds on $\dim A^N$, $N\in\N$.

\begin{definition}
Let $V^\bullet=\sum_{i\ge 1} V^i$ be a graded vector space over a ground field $k$. Assume that $\dim_k V^i<\infty$ for $i\ge 1$. \emph{Free graded commutative algebra} generated by $V$ is the algebra
\[
S^\bullet(V):=\bigotimes_{i=2k} Sym^\bullet(V^i)\otimes \bigotimes_{i=2k+1}\Lambda^\bullet(V^i),
\]
where $Sym^\bullet(V^{2k})$ and $\Lambda^\bullet(V^{2k+1})$ are symmetric and exterior graded algebras generated by $V^{2k}$ and $V^{2k+1}$ respectively.
\end{definition}

\begin{definition}
For a graded vector space $V^\bullet=\sum_{i\ge 0} V^i$ the \emph{Poincar\'e series} of $V^\bullet$ is a formal power series
\[
h(V^\bullet;t):=\sum_{i} \dim_k V^i t^i.
\]
For a graded algebra $A^\bullet$ its Poicar\'e series $h(A^\bullet;t)$ is the Poincar\'e of the underlying graded vector space.
\end{definition}

\begin{proposition}\label{prop:poincare_free}
Let $V^{\bullet}=\sum_{i\ge 1} V^i$ be a graded vector space. Let $h(V^\bullet;t)=v_1t+v_2t^2+\dots$ be the Poincar\'e series of $V^\bullet$. Then for the free graded algebra generated by $V$ one has:
\begin{equation}\label{eq:poincare_free}
h(S^\bullet(V);t)=\prod_{i=2k}(1-t^{2k})^{-v_{2k}}\prod_{i=2k+1}(1+t^{2k+1})^{v_{2k+1}}= \prod_{i}(1-(-t)^i)^{(-1)^{i+1}v_i}.
\end{equation}
\end{proposition}
\begin{proof}
For a one dimensional $V^\bullet=V^{2k}$ the corresponding free algebra $S^\bullet(V^{2k})$ is a polynomial algebra with one generator in degree $2k$, hence
\[
h(S(V^{2k});t)=1+t^{2k}+t^{4k}+\dots=\frac{1}{1-t^{2k}}.
\]
Similarly for a one dimensional $V^{2k+1}$
\[
h(S(V^{2k+1});t)=1+t^{2k+1}.
\]
These identities together with the multiplicativity of $h(\cdot;t)$ with respect to tensor products imply formula~\eqref{eq:poincare_free}.
\end{proof}

Let $\underline{s}=\{s_1,s_2,\dots\}$ be an arbitrary sequence of indeterminates. Fix some integer $N>1$ and introduce new variables $\gamma_1,\dots,\gamma_N$ such that $s_i=\sigma_i(\gamma_1,\dots,\gamma_N)$, $i\le N$ is the elementary symmetric polynomial in $\{\gamma_k\}_{k=1}^N$. We denote by $p_i(\underline{s})$ the unique presentation of the $i$-th Newton polynomial $\gamma_1^i+\dots+\gamma_N^i$ as a polynomial in $\{s_1,\dots,s_N\}$. It is easy to check that for a fixed $i$ this presentation does not depend on the choice of $N>i$. The first few polynomials $p_i(\underline{s})$ are:
\begin{equation*}
\begin{split}
p_1(\underline{s}) = & s_1\\
p_2(\underline{s}) = & s_1^2 - 2s_2\\
p_3(\underline{s}) = & s_1^3 - 3s_2s_1 + 3s_3\\
p_4(\underline{s}) = & s_1^4 - 4s_2s_1^2+4s_3s_1+2s_2^2-4s_4.\\
\end{split}
\end{equation*}
It is natural to assign gradings $\deg s_i=i$, then $p_i(\underline s)$ is homogeneous polynomial of degree $i$. We also point out that coefficient at $s_i$ in $p_i(\underline s)$ is $(-1)^{i+1}i$.
\medskip

The following lemma gives a \emph{converse} of Proposition~\ref{prop:poincare_free}.
\begin{lemma}\label{lm:mobius_inversion}
Let $V^{\bullet}=\sum_{i\ge 1} V^i$ be a graded vector space. Let $h(S^\bullet(V); t)=1+s_1 t+ s_2 t^2+\dots$ be the Poincar\'e series of the corresponding free algebra. Then the dimensions of graded components of $V^\bullet$ are given by the formulas:
\begin{equation}\label{eq:mobius_inversion}
\dim_k V^N=\frac{(-1)^{N+1}}{N}\sum_{d|N}\mu(N/d) p_d(\underline{s}),
\end{equation}
where $\underline{s}:=\{s_1,s_2,\dots\}$, $\mu(n)$ is the M\"obius function
\[
\mu(n)=
\begin{cases}
(-1)^k, &  \mbox{if $n=p_1p_2\dots p_k$, is the prme decomposition of $n$ with distinct factors};\\
0, & \mbox{otherwise},
\end{cases}
\]
and the summation is taken over all positive integers $d$ dividing $N$.
\end{lemma}
\begin{proof}
Taking logarithm of the identity~\eqref{eq:poincare_free} we get
\begin{equation}\label{eq:log1}
\log(1+s_1t+s_2t^2+\dots)=\sum_i (-1)^{i+1}v_i\log(1-(-t)^i).
\end{equation}
Now let us fix any integer $N$ and formally expand
\[
1+s_1t+\dots+s_Nt^N=\prod_{i=1}^N(1+\gamma_i t),
\]
where $\{\gamma_i\}$ are new variables. Then in $k[s_1,\dots,s_N][[t]]/t^{N+1}\subset k[\gamma_1,\dots,\gamma_N][[t]]/t^{N+1}$ we have
\[
\log(1+s_1t+s_2t^2+\dots)=\sum^N_{i=1}\log(1+\gamma_it)=\sum_{i=1}^N (-1)^{i+1}\frac{p_i(\underline{s})}{i}t^i.
\]
Taking limit $N\to\infty$, we get an identity in $k[s_1,s_2,\dots][[t]]$
\[
\log (1+s_1t+s_2t^2+\dots)=\sum_N (-1)^{N+1}\frac{p_N(\underline{s})}{N}t^N.
\]
After expanding the power series $\log(1-(-t)^i)$ in the RHS of~\eqref{eq:log1}, and comparing the coefficients of $t^N$, we conclude that
\[
\sum_{i|N} (-1)^i iv_i=-p_N(\underline{s}).
\]
Finally, with the use of M\"obius inversion formula we get
\[
v_N=\frac{(-1)^{N+1}}{N}\sum_{d|N}\mu(N/d) p_d(\underline{s}),
\]
as stated.
\end{proof}

Lemma~\ref{lm:mobius_inversion} gives a necessary condition for a sequence of integers $\{1,s_1,s_2,\dots\}$ to be the dimensions of the graded components of a free graded algebra $S^\bullet(V)$. 

\begin{corollary}\label{cor:ineq_free_alg}
If $1+s_1t+s_2t^2+\dots$ is the Poincar\'e series of a free graded algebra $S^\bullet(V)$, then for any integer $N\ge 1$ the sequence $\underline{s}=\{s_1,s_2\dots\}$ satisfies

\begin{equation}\label{ineq:free_alg}
(-1)^{N+1}\sum_{d|N}\mu(N/d) p_d(\underline{s})\ge 0.
\end{equation}

\end{corollary}

\begin{example}
For small values of $N$ Corollary~\ref{cor:ineq_free_alg} gives
\begin{enumerate}
\item[($N=1$)] $s_1\ge 0$;
\item[($N=2$)] $s_2\ge \frac{1}{2}s_1(s_1-1)$;
\item[($N=3$)] $s_3\ge \frac{1}{3}s_1(3s_2-s_1^2+1)$.
\end{enumerate}
More generally, for any $N$ inequality~\eqref{ineq:free_alg} is equivalent to a lower bound on $s_N$ of the form
\[
s_N\ge q_N(s_1,\dots,s_{N-1}),
\]
where $q_N$ is some polynomial of degree $N$, with $\deg s_i=i$.
\end{example}

\section{Toric topology}\label{sec:toric}

We begin this section by recalling the basic combinatorial notions of simplicial complexes, $f$-vectors and flag simplicial complexes.

\begin{definition}\label{def:simplicial}
An \emph{abstract simplicial complex} $\mc K$ on the set of vertices $[m]=\{1,\dots,m\}$ is a collection $\mc K=\{\sigma\}$ of subsets of $[m]$ such that for any $\sigma\in \mc K$ all subsets of $\sigma$ also belong to $\mc K$. Elements $\sigma\in\mc K$ are called simplices of $\mc K$. The \emph{dimension} of a simplex $\sigma\in \mc K$ is $|\sigma|-1$. We also assume that all one-element sets belong to $\mc K$, i.e., there are no `ghost vertices'.
\end{definition}

For an $n$-dimensional simplicial complex $\mc K$ let $f_i$ be the number of its $i$-dimensional simplices. Then the \emph{$f$-vector} of simplicial complex $\mc K$ is $(f_0,\dots,f_n)$.

\begin{definition}\label{def:flag}
Let $\mc K$ be a simplicial complex on the set of vertices $[m]$. A subset $\sigma\subset[m]$ is a \emph{minimal non-face} of $\mc K$ if $\sigma\not\in\mc K$, but all proper subsets $\sigma'\subsetneq\sigma$ are simplices of $\mc K$: $\sigma'\in\mc K$.

Simplicial complex $\mc K$ is \emph{flag} if all its minimal non-faces are two-element subsets of $[m]$. Clearly, any flag simplicial complex is determined by its 1-skeleton. 
\end{definition}

Now we describe the construction originating from \emph{toric topology}~\cite{bu-pa-15}. This construction associates to any simplicial complex $\mc K$ and a pair of topological spaces $(X,A)$ a topological space $(X,A)^\mc K$. The main idea is that `topology' of $(X,A)^\mc K$ is somehow captures `combinatorics' of $\mc K$.

\begin{definition}[Polyhedral product]\label{def:dj}
Let $\mc K$ be an abstract simplicial complex on the set of vertices $[m]$. Let $(X, A)$ be a pair of topological spaces. The \emph{polyhedral product} associated to $\mc K$ is a topological space $(X,A)^\mc K\subset X^m$ defined as follows. For any subset $\sigma\subset [m]$ let us denote a `building block' inside $X^m$:
\[
(X,A)^\sigma:=\prod_{i\in \sigma} X\times \prod_{i\not\in\sigma} A\subset X^m.
\]
Then by definition $(X,A)^\mc K$ is the union of the building blocks $(X,A)^\sigma$, where $\sigma$ runs over simplices of $\mc K$:
\[
(X,A)^\mc K:=\bigcup_{\sigma\in\mc K}(X,A)^\sigma.
\]
\end{definition}

\begin{example}\label{ex:def_zk_dj}
Let us consider a pair $(D^2,S^1)$, where $D^2$ is the unit disc and $S^1$ is its boundary. The corresponding $\mc K$-power is called the \emph{moment-angle-complex} and denoted by 
\[
Z(\mc K):=(D^2,S^1)^\mc K.
\]

Starting with the pair $(\C P^\infty, pt)$, we get the \emph{Davis-Januszkiewicz space} 
\[
DJ(\mc K)=(\C P^\infty, pt)^\mc K.
\] 
It is known (see~\cite[Theorem\,4.3.2]{bu-pa-15}) that the moment-angle-complex $Z(\mc K)$ is a homotopy fibre of the inclusion $DJ(\mc K)\to (\C P^\infty)^m$. In particular, the higher homotopy groups of $Z(\mc K)$ and $DJ(\mc K)$ coincide:
\[
\pi_i(Z(\mc K))\simeq \pi_i(DJ(\mc K)),\quad i\ge 3.
\]
%
\end{example}

\begin{remark}
The crucial property of Davis-Januszkiewicz space is that for any simplicial complex $\mc K$ the cohomology ring of $DJ(\mc K)$ is the face ring:
\[
H^*(DJ(\mc K); k)\simeq k[\mc K]:=k[v_1,\dots,v_m]/\mc I_{SR}(\mc K), \deg v_i=2,
\] 
where $\mc I_{SR}$ is the ideal generated by monomials corresponding to all non-simplices of $\mc K$: 
\[
\mc I_{SR}=\{v_{i_1}\dots v_{i_k}|\ \sigma=\{i_1,\dots,i_k\}\not\in\mc K\}.
\]
\end{remark}

The following result is due to Panov and Ray. 
\begin{proposition}{\cite[Prop.\,9.5]{pa-ra-08}}\label{prop:pa-ra}
For any flag simplicial complex $\mc K$, we have
\[
h\bigl(H^\bullet(\Omega DJ(\mc K); \Q); t\bigr)=h(\Q[\mc K], (-t)^{1/2})^{-1}=
\Biggl(1+ \sum_{i=0}^{\dim \mc K} (-1)^{i+1}f_i\frac{t^{i+1}}{(1+t)^{i+1}} \Biggr)^{-1}.
\]
\end{proposition}

Let $X$ be an arbitrary simply connected pointed $CW$-complex. Its loop space $\Omega X$ is an $H$-space, hence its rational cohomology is a Hopf algebra. At the same time, any Hopf algebra over a field of characteristic zero is a free graded algebra, see~\cite{du-fo-no-90}:
\[
H^\bullet(\Omega X,\Q)\simeq S^\bullet(V),
\]
where $V=V^\bullet$ is a graded vector space spanned by primitive elements in Hopf algebra $H^\bullet(\Omega X,\Q)$. Alternatively, $V^\bullet$ could by described as the dual to $\pi_\bullet(\Omega X)\simeq \pi_{\bullet}(X)[-1]$, where $[-1]$ stands for the grading shift, see, e.g.,~\cite{fe-op-ta-08}. This observation together with remark in Example~\ref{ex:def_zk_dj} explain that the following slightly reformulated result due to Denham and Suciu contains essentially the same information as Proposition~\ref{prop:pa-ra}:
\begin{theorem}{\cite[Theorem\,4.2.1]{de-su-07}}
Let $\mc K$ be a flag complex with the face ring $k[\mc K]$. Then the ranks $\pi_i=\pi_i(Z(\mc K))$ of the homotopy groups of the moment-angle complex $Z(\mc K)$ are given by
\[
\prod_{r=1}^\infty 
\frac{
	(1+t^{2r-1})^{\pi_{2r}}
}
{
	(1-t^{2r})^{\pi_{2r+1}}
} =
h({\rm Tor}(\Q[\mc K], \Q), (-t)^{1/2}, -(-t)^{1/2})^{-1},
\]
where $h(A^\bullet, t_1, t_2)$ is the bigraded Poincar\'e series of $\rm Tor$-algebra.
\end{theorem}

\section{Proof of the main result}\label{sec:main}
Now we are ready to prove our main result:
\begin{theorem}\label{thm:main}
Let $\mc K$ be a flag complex with $f$-vector $(f_0,\dots,f_n)$.  Then the coefficients of the power series $\Biggl(1+ \sum_{i=0}^{\dim \mc K} (-1)^{i+1}f_i\frac{t^{i+1}}{(1+t)^{i+1}} \Biggr)^{-1}$ satisfy inequalities~\eqref{ineq:free_alg}. Specifically, for any $N\ge 1$ we have
\begin{equation}\label{ineq:main}
(-1)^{N}\sum_{d|N}\mu(N/d) (-1)^d p_d(\underline{\alpha})\ge 0,
\end{equation}
where $p_d$ is the $d$-th Newton polynomial expressed in elementary symmetric polynomials $\underline{\alpha}=(\alpha_1,\alpha_2,\dots)$ with
\[
\alpha_n:=\sum_{i=0}^{n-1} f_i\binom{n-1}{i}.
\]
\end{theorem}

\begin{proof}
The fact that the coefficients of a power series $Q(t):=\Biggl(1+ \sum_{i=0}^{\dim \mc K} (-1)^{i+1}f_i\frac{t^{i+1}}{(1+t)^{i+1}} \Biggr)^{-1}$ satisfy inequalities~\eqref{ineq:free_alg} immediately follows from the fact that this power series is the Poincar\'e series of a free graded algebra $H^*(\Omega DJ(\mc K),\Q)$ and from Corrolary~\ref{cor:ineq_free_alg}.

To find the explicit form of these inequalities we follow the proof of Lemma~\ref{lm:mobius_inversion} with slight modifications. Namely, we first rewrite $Q(t)$ as follows
\[
Q(t)=\Biggl(\sum_{n=0}^\infty \sum_{j=0}^{n-1} f_j\binom{n-1}{j}(-t)^n\Biggr)^{-1}.
\]
Now we take logarithm of the identity
\[
Q(t)=\prod_{r=1}^\infty 
\frac{
	(1+t^{2r-1})^{v_{2r}}
}
{
	(1-t^{2r})^{v_{2r+1}}
}
\]
and comparing coefficients at $t^N$ as in the proof of Lemma~\ref{lm:mobius_inversion} we find that
\[
\sum_{i|N}(-1)^iiv_i=-p_N(\underline{\{(-1)^k\alpha_k\}}^\infty_{k=1})=(-1)^{N+1} p_N(\underline{\alpha}).
\]
Applying again the M\"obius inversion formula we get the identity
\[
v_N=\frac{(-1)^N}{N}\sum_{d|N}\mu(N/d) (-1)^d p_d(\underline{\alpha}),
\]
which implies the stated inequality.
\end{proof}

\begin{example}
For small values of $N$ Theorem~\ref{thm:main} gives:
\begin{enumerate}
\item[($N=1$)] Inequality~\eqref{ineq:main} reads $p_1\ge 0$. Plugging $p_1=f_0$, we get 
\[
f_0\ge 0;
\]
\item[($N=2$)] Inequality~\eqref{ineq:main} reads $p_2+p_1\ge 0$. Plugging $p_2=\alpha_1^2-2\alpha_2$ with $\alpha_1=f_0$, $\alpha_2=f_0+f_1$, we get 
\[
f_1\le\binom{f_0}{2};
\]
\item[($N=3$)] Inequality~\eqref{ineq:main} reads $p_3-p_1\ge 0$. With $p_3=\alpha_1^3-3\alpha_1\alpha_2+3\alpha_3$, $\alpha_1=f_0$, $\alpha_2=f_0+f_1$, $\alpha_3=f_0+2f_1+f_2$ we arrive at
\[
f_0^3-3f_0(f_0+f_1)+3f_2-f_0\ge 0,
\]
or, equivalently,
\[
f_2\ge \binom{f_0}{3}-(f_0-2)\biggl(\binom{f_0}{2}-f_1\biggr).
\]
\end{enumerate}
The first two inequalities are trivially satisfied for any simplicial complex and simply indicate that the number of vertices is non-negative and the number of edges is bounded above by $\binom{f_0}{2}$. 

The third inequality is not satisfied for a general simplicial complex. It says that for a flag simplicial complex the number of non-triangles is less or equal then the number of non-edges times $(f_0-2)$. Indeed, each non-triangle contains at least one non-edge, while every non-edge is a side of $f_0-2$ non-triangles.
\end{example}

\begin{remark}
For a general $N$ inequality~\eqref{ineq:main} has form
\[
(-1)^Nf_N\ge q_N(f_1,\dots, f_{N-1}),
\]
where $q_N$ is a polynomial of degree $N$ with $\deg f_i=i+1$.
\end{remark}

It is interesting if inequalities~\eqref{ineq:main} has a simple combinatorial interpretation:

\begin{question}
Does there exist a direct combinatorial proof of inequalities~\eqref{ineq:main} for all $N$?
\end{question}

\section{Acknowledgements}

I am grateful to Taras Panov for fruitful discussions and many useful remarks.

\bibliographystyle{siam}
\bibliography{biblio}

\end{document}